\documentclass[12pt]{amsart}
\usepackage{amscd,amssymb,url}
\usepackage[graph,frame,poly,arc]{xy}  
\usepackage[plainpages,backref,urlcolor=blue]{hyperref}
\usepackage{verbatim}

\usepackage{tikz}
\usetikzlibrary{calc}

\theoremstyle{plain}
\newtheorem{thm}[subsection]{Theorem}
\newtheorem{lem}[subsection]{Lemma}
\newtheorem{prop}[subsection]{Proposition}
\newtheorem{cor}[subsection]{Corollary}

\theoremstyle{definition}
\newtheorem{rk}[subsection]{Remark}

\newtheorem{ex}[subsection]{Example}

\numberwithin{equation}{section}
\newcommand{\A}{{\mathcal A}}
\newcommand{\B}{{\mathcal B}}

\newcommand{\D}{{\mathcal D}}
\newcommand{\CC}{{\mathcal C}}

\newcommand{\PPP}{{\mathcal P}}

\newcommand{\R}{\mathbb{R}}
\newcommand{\C}{\mathbb{C}}
\newcommand{\PP}{\mathbb{P}}

\begin{document}


\title [On some free line arrangements]
{Free line arrangements with low maximal multiplicity}

\author[Alexandru Dimca]{Alexandru Dimca}
\address{Universit\'e C\^ ote d'Azur, CNRS, LJAD, France and Simion Stoilow Institute of Mathematics,
P.O. Box 1-764, RO-014700 Bucharest, Romania.}
\email{Alexandru.Dimca@univ-cotedazur.fr}

\author[Lukas K\"uhne]{Lukas K\"uhne}
\address{Universit\"at Bielefeld, Fakult\"at f\"ur Mathematik, Bielefeld, Germany.}
\email{lkuehne@math.uni-bielefeld.de}

\author[Piotr Pokora]{Piotr Pokora}
\address{Department of Mathematics,
University of the National Education Commission Krakow,
Podchor\c a\.zych 2,
PL-30-084 Krak\'ow, Poland.}
\email{piotrpkr@gmail.com, piotr.pokora@uken.krakow.pl}



\begin{abstract} 
Let $\A$ be a free arrangement of $d$ lines in the complex projective plane, with exponents $d_1\leq d_2$. Let $m$ be the maximal multiplicity of points in $\A$. In this note, we describe first the simple cases $d_1 \leq m$. Then we study  the case $d_1=m+1$, and describe which line arrangements can occur by deleting or adding a line to $\A$.
 When $d \leq 14$, there are only two  free arrangements with $d_1=m+2$, namely one with degree $13$ and the other with degree $14$. We study their geometries in order to deepen our understanding of the structure of free line arrangements in general.

\end{abstract}
 
\maketitle

\section{Introduction}

Let $S=\C[x,y,z]$ be the graded polynomial ring in the variables $x,y,z$ with the complex coefficients. Let $\A : f=0$ be a line arrangement of degree $d$ in the complex projective plane $\PP^2$. Denote by $J_{f}$ the Jacobian ideal generated by all partial derivatives of $f$, and we define the Milnor algebra as $M(f) = S/J_{f}$. Recall that if $\A : f=0$ is a line arrangement in $\mathbb{P}^{2}$, then $M(f)$ admits the following minimal resolution: 
$$0 \rightarrow \bigoplus_{i=1}^{s-2}S(-e_{i}) \rightarrow \bigoplus_{i=1}^{s}S(1-d - d_{i}) \rightarrow S^{3}(1-d)\rightarrow S \rightarrow M(f) \rightarrow 0,$$
where $e_{1} \leq e_{2} \leq \ldots \leq e_{s-2}$ and $1\leq d_{1} \leq \ldots \leq d_{s}$, see for instance \cite{HS}. The integers $d_1 \leq \ldots \leq d_s$ are called the exponents of $\A$. When $s=2$, then $\A$ is called \textit{free}. If $\A$ is a free line arrangement and $\overline \A$ is another line arrangement with the same intersection lattice, then
Terao's Conjecture predicts that $\overline \A$ is also free.

If $\A$ is a line arrangement and $p$ is a point in $\A$, then the multiplicity $m_p$ of $p$ in $\A$ is by definition the number of lines in $\A$ containing the point $p$. Let $m=m(\A)$ be the maximal multiplicity of the intersection points in $\A$. It is known that either $d_1=d-m$, or
$$m-1 \leq d_1 <d-m,$$
see \cite{Mich}. Moreover, if $d_1 \in \{m-1,m\}$, then Terao's Conjecture holds for $\A$, see \cite[Remark 5.4]{Der} and Propositions \ref{prop00} and \ref{prop01}, where these two cases are briefly discussed.
The same holds when $d_1=d-m$, see \cite{Mich}.

From now on we consider free line arrangements $\A$ such that
$$d_1=m+ \epsilon <d-m,$$
for some integer $\epsilon >0$.
Note that if $\overline \A$ is another line arrangement with the same intersection lattice as $\A$ and $\overline \A$ is not free, then we have
$$m \leq \overline d_1 < m+ \epsilon,$$
where $\overline d_1$ is the minimal exponent of $\overline \A$.
{\it It follows that the bigger $\epsilon$ is, the more space there is for counterexamples to Terao's Conjecture.} 

In Section~\ref{sec2}, we discuss the case $\epsilon =1$, see Theorems \ref{thm02} and \ref{thm03}, which yields a precise description of the behavior of a free arrangement with $d_1=m+1$ when deleting or adding a line.

Then, in Section~\ref{sec3}, we consider the case $\epsilon \geq 2$, in which free arrangements with $d_1 = m + \epsilon$ appear to be rather rare.
First we show in Proposition \ref{bound} a new bound on $d=|\A|$ in terms of $m$ and $\epsilon$. In particular, this shows that $d \geq 15$ when $\epsilon=3$, see Remark \ref{rkE3}.
For $\epsilon =2$ and $d \leq 14$, there are only two such arrangements, $\A$ (resp. $\CC$) containing $d=13$ (resp. $d=14$) lines, and we study them in the remaining part of Section 3.
The fact that there are only two such free arrangements in this range, as well as their equations described below, follows from the article~\cite{BBJKL21} and the associated database \cite{BK} constructed to verify Terao's Conjecture up to $d\le 14$~\cite{BK23}.
Starting with their equations, we give geometric proofs (i.e. without using computer-aided computations) for their freeness, see Theorems \ref{thm2} and \ref{thm20}. 

Although the freeness of the arrangements $\A$ and $\CC$ can be established more directly by observing that they are divisionally free (see Remark~\ref{rkDF}), we have chosen to present the longer proofs given in Theorems~\ref{thm2} and~\ref{thm20}. These arguments rely on more general techniques and may therefore be applicable to similar arrangements that are not necessarily divisionally free.

We hope that further study of such examples will allow us to obtain results for the cases $\epsilon > 1$ analogous to those established for $\epsilon = 1$.

\section{On free line arrangements with \texorpdfstring{$m-1 \leq d_1 \leq m+1$}{}}\label{sec2} 
 For integers $d > d_1>0$, following \cite{dPCTC}, we define
$$\tau_{max}(d,d_1)=(d-1)^2-d_1(d-d_1-1).$$
Let $\A$ be an arrangement of $d$ lines in $\PP^2$ and let $d_1\leq d_2$ be the smallest two exponents of $\A$. It is well-known, see for instance \cite{Dmax,dPCTC}, that $\A$ is free if and only if
\begin{equation}
\label{eqCTC}
\tau(\A)=\tau_{max}(d,d_1),
\end{equation}
where $\tau(\A)$ is the global Tjurina number of the arrangement $\A$.
Say $n_r$ is the number of intersection points in $\A$ of multiplicity $r$.
The Tjurina number $\tau(\A)$ can be combinatorially computed as
\[
\tau(\A) = \sum_{r\ge 2}n_r(r-1)^2.
\]
Since the exponent $d_1$ is not determined by the combinatorics, we try in this section to describe results similar to the above one, but involving
easy combinatorial properties of $\A$. 
Let $m=m(\A)$ be the maximal multiplicity of points of $\A$. If $d-m \leq m-1$, then it is known that
$$d_1=d-m$$
see \cite{Mich}. Hence we can apply \eqref{eqCTC} to obtain an obvious characterization of free arrangements.
In the remaining cases, when $d+1 >2m$, it is known that
$$m-1 \leq d_1 \leq d_2 \leq d-m,$$
and the equality $m-1=d_1$ implies that the line arrangement $\A$ is free with exponents
$(d_1,d_2)=(m-1,d-m)$, see again \cite{Mich}. In fact, we have the following more precise result, see \cite[Theorem 1.12  (1)]{Der}.

Recall that a point $p \in A$ of multiplicity $m_p \geq 2$  is a modular point for the line arrangement $\A$ if for any other point 
$q \in \A$ of multiplicity $m_q \geq 2$, the line determined by the points $p$ and $q$ is in $\A$. A line arrangement having a modular point is called  supersolvable, see \cite{OT}.
 \begin{prop}
\label{prop00}
Assume that $\A \subset \mathbb{P}^{2}$ is a line arrangement satisfying
$$\tau(\A)=\tau_{max}(d,m-1)$$
where $d=|\A|$ and $m=m(\A)$. Then $\A$ is a free line arrangement with exponents
$(d_1,d_2)=(m-1,d-m)$. In addition, $\A$ is a supersolvable line arrangement and any point $p$ in $\A$ with multiplicity $m$ is a modular point for $\A$.
\end{prop}

It follows that the interesting cases for understanding the structure of free line arrangements are described by the inequalities
$$m \leq d_1 \leq d_2 \leq d-m.$$
One has the following result.
 \begin{prop}
\label{prop01}
Assume that $\A$ is a line arrangement satisfying
$$\tau(\A)=\tau_{max}(d,m)$$
where $d=|\A|$ and $m=m(\A)$. Then $\A$ is a free line arrangement with exponents
$(d_1,d_2)=(m,d-m-1)$.
\end{prop}
\proof
We know that $d_1 \geq m-1$. The equality $d_1=m-1$ would give by the above discussion
$\tau(\A)=\tau_{max}(d,m-1).$ However, a simple computation shows that
$$\tau_{max}(d,m-1) \ne \tau_{max}(d,m),$$
and hence $d_1>m-1$.
On the other hand, the inequality $d_1 >m$ would imply
$$\tau(\A) <\tau_{max}(d,m),$$
see \cite{dPCTC}. Hence the only possibility is $d_1=m$, and the claim follows from our discussion above.
\endproof
 \begin{rk}
\label{rk01}
Assume that $\A$ is a line arrangement satisfying $d_1=m=m(\A)$ and let $p \in \A$ be a point of multiplicity $m$.
Then it follows from  \cite[Theorem 1.12  (2)]{Der} that all the multiple points in $\A$ not connected to $p$ by lines in $\A$ are situated on a line $L$. The example of the arrangement
$$\A: x(x^{m-1}-y^{m-1})(y^{m-1}-z^{m-1})(x^{m-1}-z^{m-1})=0,$$
where $p=(0:0:1)$, shows that $L$ may not be unique. Indeed, the only point not connected to $p$ is $q=(1:0:0)$ and hence any line $L$ with $q \in L$ will do the job. The line $L_y: y=0$ and the line $L_z: z=0$ are special, since when we add $L_y$ or $L_z$ to $\A$ we get a supersolvable arrangement. They are characterized by the fact that $q \in L$ and the number of points in the intersection $L \cap \A$ is minimal, namely $m+1$. 

{\it It is an interesting open question whether the line $L$ in 
 \cite[Theorem 1.12  (2)]{Der} can be always chosen such that $\A \cup L$ is supersolvable. This will give a strong information on the structure of line arrangements satisfying $d_1=m(\A)$, namely they are obtained from a supersolvable line arrangement by deleting a suitable line.}
\end{rk}

We need the following result to continue our discussion.

\begin{lem}
\label{lem01}
Let $\A$ be a line arrangement, $L \in \A$ be a line and denote by $\A'$ the line arrangement obtained from $\A$ by deleting the line $L$.
Then
$$\tau(\A)-\tau(\A')=2(d-1)-r_L$$
where $d=|\A|$ and $r_L=|L \cap \A'|$.
 \end{lem}
 \proof
 For any point $q \in \A' \cap L$, denote by $m_q'\geq 1$ its multiplicity in $\A'$. 
 One has
 $$\tau(\A)-\tau(\A')=\sum_{q \in L \cap \A'}(m_q'^2-(m_q'-1)^2)=\sum_{q \in L \cap \A'}(2m_q'-1)=2(d-1)-r_L,$$
 since $\sum_{q \in L \cap \A'}m_q'=\deg(\A')=d-1$ by B\'ezout's Theorem.
 \endproof
 
We now consider line arrangements satisfying $d_1=m(\A)+1$. Recall that an arrangement $\mathcal{A} \subset \mathbb{P}^{2}$ of $d$ lines is \textit{plus-one generated} if $s=3$ and $d_{1}+d_{2}=d$

 \begin{thm}
\label{thm02}
Assume that $\A$ is a line arrangement satisfying
$$\tau(\A)=\tau_{max}(d,m+1)$$
where $d=|\A|$ and $m=m(\A)\leq \frac{d-3}{2}$. Let $L \in \A$ and let $\A'$ be the line arrangement obtained from $\A$ by deleting the line $L$.
If $m=m(\A')$, then one of the following situations occurs.
\begin{enumerate} 

\item $r_L=2(d-1)-3m$ and $\A'$ is free with exponents $(m-1,d-m-1)$. In this case $3m \geq d-1$ and the arrangement $\A$ is plus-one generated with exponents $(m,d-m, r_L-1)$.

\item $r_L=d-m-1$ and $\A'$ is free with exponents $(m,d-m-2)$. In this case $\A$ is free with exponents $(m+1,d-m-2)$ .

\item In all other cases one has $r_L <d-1-m$. In particular, if in addition $\A$ is free with
exponents $(m+1,d-m-2)$, then either $m+1<d-m-2$ and $\A'$ is free with exponents $(m+1,d-m-3)$, or $m+1 \leq d-m-2$ and $\A'$ is a plus-one generated arrangement with exponents $(m+1,d-m-2, d-r_L)$.
\end{enumerate} 
\end{thm}
\proof
The equality $\tau(\A)=\tau_{max}(d,m+1)$ implies $d_1 \leq m+1$.
On the other hand, the discussion just before Proposition \ref{prop00} implies that $d_1 \geq m$.
Let $d_1'$ be the minimal exponent of $\A'$. It is clear that 
$$ m-1 \leq d_1-1 \leq d_1' \leq d_1 \leq m+1.$$
Assume first that  $d_1'=m-1$. Then as seen above
$\A'$ is free and hence
$$\tau(\A')= (d-2)^2-(m-1)(d-m-1),$$
which gives the equality
$$\tau(\A)-\tau(\A')=3m.$$
 Lemma \ref{lem01} implies that $r_L=2(d-1)-3m$ and since $r_L \leq d-1$ we get $3m \geq d-1$.
To see that the arrangement $\A$ is plus-one generated with exponents $(m,d-m,r_L-1)$, we apply \cite[Theorem 1.4 (3)]{POG}.

Assume from now on that $d_1' \geq m$. By \cite{dPCTC} we have
$$\tau(\A') \leq \tau_{max}(d-1,m),$$
and the equality holds if and only if $\A'$ is free with exponents $(m,d-m-2)$. 
It follows that
$$\tau(\A)-\tau(\A') \geq \tau_{max}(d,m+1)-\tau_{max}(d-1,m)=d+m-1,$$
and the equality holds if and only if $\A'$ is free with exponents $(m,d-m-2)$. Lemma \ref{lem01} tells us that this equality is equivalent to the equality $r_L=d-1-m$. To see that in these conditions $\A$ is free with exponents $(m+1,d-m-2)$, we apply \cite[Theorem 1.4 (2)]{POG}.

The last claim in $(3)$ follows from  \cite[Theorem 1.3]{POG}, using either claim $(1)$ or claim $(3)$ there.
\endproof

 \begin{ex}
\label{ex02}
(i) Consider the arrangement
$$\A: f=yz(x^4-y^4)(y^4-z^4)(x^4-z^4)(x-2y)=0$$
and the line $L:x-2y=0$. Then this setting provides an example for Theorem \ref{thm02} $(1)$ above, where $d=15$, $m=6$, $r_L=10$.

(ii) Consider the line arrangement
$$\A: f=xyz(x+y)(x-y)(x+z)(x-z)(2y-z)(x+2y-z)(x-2y+z)(y-z)=0,$$
which is free with exponents $(5,5)$. It has $n_2=13$, $n_3= 2$ and
$n_4=6$. The line $L:z=0$ is such that $|L \cap \A'|=6$, the maximal value for such intersection sets, where $\A'$ is obtained from $\A$ by deleting $L$. 
 Similarly, consider the line arrangement
$$\A: f=xyz(x+y)(x+ey)(x+z)(x+ez)(y-ez)$$
$$(x+ey-e^2z)(x+e^2y+z)(y-z)(x-ey+ez) =0,$$
where $e^2-e+1=0$,
which is free with exponents $(5,6)$. It has $n_2=9$, $n_3= 7$ and
$n_4=6$. The line $L_{12}:x-ey+ez=0$ is such that $|L_{12} \cap \A'|=7$, the maximal value for such intersection sets in this case, where again $\A'$ is obtained from $\A$ by deleting $L$.
Hence both these examples corresponds to the claim in Theorem \ref{thm02} $(2)$ above.

(iii) Note that the monomial arrangement
$$\A: (x^m-y^m)(y^m-z^m)(x^m-z^m)=0$$
is free with exponents $(m+1,2m-2)$ where $m=m(\A)$, and for any line $L \in \A$ one has $r_L=m+1<d-m-1$, as in Theorem \ref{thm02} $(3)$. The corresponding arrangement
$\A'$ in this case is nearly free, a special class of plus-one generated arrangements, with exponents $(m+1,2m-2,2m-2)$,
as it follows from \cite[Theorem 1.3 (3)]{POG}.

 \end{ex}
 
  \begin{thm}
\label{thm03}
Assume that $\A$ is a line arrangement satisfying
$$\tau(\A)=\tau_{max}(d,m+1),$$
where $d=|\A|$ and $m=m(\A)\leq \frac{d-3}{2}$. Let $p \in \A$ be a point of multiplicity $m$ and $L$ a line passing through $p$ which is not in $\A$. Let $\B$ be the arrangement obtained by adding the line $L$ to $\A$. Then one of the following holds.
\begin{enumerate} 

\item $r_L=|L \cap \A|=3(m+1)-d$ and $\B$ is a free line arrangement with exponents $(m,d-m)$. In this case $m \geq (d-2)/3$ and the line arrangement $\A$ is a plus-one generated arrangement with exponents $(m,d-m,2d-3(m+1))$.

\item $r_L=|L \cap \A|=m+2$
and  $\B$ is free with exponents $(m+1,d-m-1)$. Then $\A$ is free with exponents $(m+1,d-m-2)$.

\item In all the other cases $r_L >m+2$. In particular, if in addition $\A$ is free with exponents $(m+1,d-m-2)$, then 
either $m+1<d-m-2$, $r_L=d-m-1$ and $\B$ is free with exponents $(m+2,d-m-2)$, or
$\B$ is a plus-one generated arrangement  with exponents $(m+2,d-m-1, r_L-1)$.

\end{enumerate} 
\end{thm}

\proof
Note that $m(\B)=m+1$ and hence the minimal exponent of $\B$, say
$d_1''$, satisfies $d_1'' \geq m$. If $d_1''=m$, then it follows that $\B$ 
is a free line arrangement with exponents $(m,d-m)$ and the formula for $r_L$ comes from Lemma \ref{lem01}. Since $r_L \geq 1$, we get 
$m \geq (d-2)/3$. The claim about $\A$ follows from  \cite[Theorem 1.3 (3)]{POG} by observing that the cases $(1)$ and $(2)$ of that result cannot hold in our situation. 

Consider now the case $d_1'' \geq m+1$. Then \cite{dPCTC} implies that
$\tau(\B) \leq \tau_{max}(d+1,m+1)$ and hence
$$\tau(\B)-\tau(\A) \leq  \tau_{max}(d+1,m+1)- \tau_{max}(d,m+1)=2d-m-2.$$
This yields the first claim in $(2)$ as well as the claims about $r_L$.
The claim about $\A$ follows from  \cite[Theorem 1.3 (1)]{POG}.

The last claim in $(3)$ follows from  \cite[Theorem 1.4]{POG}, using either the claim $(2)$ or the claim $(3)$ there.
\endproof

 \begin{ex}
\label{ex03}
Consider the monomial line arrangement
$$\A: (x^m-y^m)(y^m-z^m)(x^m-z^m)=0$$
for $m \geq 3$, which is free with exponents $(d_1,d_2)=(m+1,2m-2)$ and $m=m(\A)$. If we take $p=(0:0:1)$ and $L:x=0$ then $r_L=|L \cap \A |=m+2$. Hence this is an example of situation as described in Theorem \ref{thm03} $(2)$ above.
For the same arrangement $\A$, if we take now $L: x+2y=0$, then
$$r_L=|L \cap \A |=2m+1>d_2+1=2m-1.$$
Hence this gives now an example of situation as described in Theorem \ref{thm03} $(3)$ above, with $\B$ a plus-one generated arrangement with exponents
$(m+2,2m-1,2m)$.
{\it It is an open question whether the case $(1)$ in Theorem \ref{thm03} can really occur.}
 \end{ex}
 
  \begin{rk}
\label{rk03}
(i) Note that the value of $r_L$ in Theorem \ref{thm02} is determined by the intersection lattice of $\A$ since $L \in \A$. On the other hand, the value of $r_L$ in Theorem~\ref{thm03} may not be determined by the intersection lattice of $\A$ since $L \notin \A$. This may lead to a potential fail of Terao's Conjecture.

(ii) In Theorem \ref{thm02}, one has
$$2(d-1)-3m <d-m-1,$$
and hence the case $(1)$, $(2)$, $(3)$ correspond to decreasing values of $r_L$. On the other hand, in Theorem \ref{thm03} one has
$$3(m+1)-d < m+2,$$
and hence the case $(1)$, $(2)$, $(3)$ correspond to increasing values of $r_L$.

 \end{rk}
 
 A simpler and complete description of the line arrangements satisfying
 $\tau(\A)=\tau_{max}(d,m+1)$ similar to Propositions \ref{prop00} and \ref{prop01} is given in the following result.
 
  \begin{thm}
\label{thmAe1}
Assume that $\A$ is a line arrangement satisfying
$$\tau(\A)=\tau_{max}(d,m+1)$$
where $d=|\A|$ and $m=m(\A)\leq \frac{d-3}{2}$. 
Then one of the following situations occurs.
\begin{enumerate} 
\item The arrangement $\A$ is plus-one generated with exponents $(m,d-m, 2d-3m-3)$.

\item The arrangement $\A$ is free with exponents $(m+1,d-m-2)$.
\end{enumerate} 
\end{thm}

\proof
The equality $\tau(\A)=\tau_{max}(d,m+1)$ implies that
$$d_1 \leq m+1.$$
On the other hand, we know that $d_1 \geq m-1$ and the equality
$d_1=m-1$ is not possible by Proposition \ref{prop00} since
$$\tau(\A)=\tau_{max}(d,m+1)<\tau_{max}(d,m-1).$$
If $d_1=m$, then we can apply  \cite[Theorem 1.12 (2)]{Der}. 
In addition, since
$$\tau(\A)=\tau_{max}(d,m+1)<\tau_{max}(d,m),$$
the arrangement $\A$ cannot be free with exponents $(m,d-m-1)$. It follows that $\A$ is a plus-one generated arrangement with exponents
$(d_1,d_2,d_3)$. In order to determine $d_3$ we use the formula
$$\tau(\A)=(d-1)^2-d_1(d-1-d_1)-(d_3-d_2+1),$$
which comes from \cite[Proposition 2.1 (4)]{3syz}, using the relation
$d_1+d_2=d$ and then we replace $\tau(\A)$ by $\tau_{max}(d,m+1)$.
This ends the proof for case (1). 

For case (2), if $d_1=m+1$, we get that $\A$ is free with exponents $(m+1,d-m-2)$ by \cite{Dmax,dPCTC}.

\endproof

  \begin{cor}
\label{corAe1}
Assume that $\A$ is a line arrangement satisfying
$$\tau(\A)=\tau_{max}(d,m+1)$$
where $d=|\A|$ and $m=m(\A)\leq \frac{d-3}{2}$. Let $L \in \A$ be any line, and let $\A'$ be the line arrangement obtained from $\A$ by deleting the line $L$. We set $r_L=|L \cap \A'|$, that is $r_L$ is the number of multiple points of $\A$ on the line $L$.
Then one of the following situations occurs.
\begin{enumerate} 

\item The arrangement $\A$ is plus-one generated with exponents $(m,d-m, 2d-3m-3)$, $\A'$ is free with exponents $(m-1,d-m-1)$, $3m\geq d-1$, and 
$$r_L=2d-3m-2.$$

\item The arrangement $\A$ is free with exponents $(m+1,d-m-2)$, $\A'$ is free with exponents $(m,d-m-2)$ and
$$r_L  =d-m-1.$$

\item The arrangement $\A$ is free with exponents $(m+1,d-m-2)$, $\A'$ is free with exponents $(m+1,d-m-3)$,
$$m <\frac{d-3}{2} \text{ and }   r_L=m+2 <d-m-1.$$

\item The arrangement $\A$ is free with exponents $(m+1,d-m-2)$, $\A'$ is a plus-one generated arrangement with exponents
$(m+1,d-m-2,d-r_L)$,
$$m \leq \frac{d-3}{2} \text{ and }   r_L <d-m-1.$$

\end{enumerate} 
\end{cor}
\proof
We use Theorem \ref{thm02} and we note first that under our assumptions one has
$$2(d-1)-3m >d-m-1.$$
It follows the cases (1) and (2) in Corollary \ref{corAe1} correspond to cases (1) and (2) in Theorem \ref{thm02}.

In case (3) in Theorem \ref{thm02}, if we assume that $\A$ is free with exponents $(m+1,d-m-2)$, then for
$$m<\frac{d-3}{2}$$
we get that $\A'$ is free with exponents $(m+1,d-m-3)$,
and for 
$$m \leq \frac{d-3}{2}$$
$\A'$ may also be a plus-one generated arrangement with exponents
$(m+1,d-m-2,d-r_L)$. 
These two cases yield the cases (3) and (4) in Corollary \ref{corAe1}. Since both arrangements $\A$ and $\A'$ are free in case (3), it follows that $\A$ is divisionally free, and hence $r_L=m+2<d-m-1$ in this case.

\endproof

 \begin{ex}
\label{exAe1}

(i) To get examples of the case (2) in Corollary \ref{corAe1}, one may consider the line arrangement $\A=\B_{12}$ with the line $L=L_4$ (resp. $\A=\CC'$ with the line $L=L_4$)
discussed in the proof of Theorem \ref{thm2} (resp. Theorem \ref{thm20} below. In this example we have $(d,m)=(12,4)$ (resp. $(d,m)=(13,4)$).

 (ii) Consider the line arrangement
$$\A: f=xyz(x+y)(x-y)(x+z)(x-z)(2y-z)(x+2y-z)(x-2y+z)(y-z)=0,$$
which is free with exponents $(5,5)$ and has $m=m(\A)=4$, see Example \ref{ex02} above.
The line $L:x=0$ is such that $|L \cap \A'|=5<6=d-m-1$, where $\A'$ is obtained from $\A$ by deleting $L$. A direct computation shows that
$\A'$ is a plus-one generated arrangement with exponents
$(5,5,6)$, and therefore this situation provides an example of the case (4) in Corollary \ref{corAe1} where the equality 
$$m = \frac{d-3}{2}$$
holds.

(iii) {\it Mystic Pentagram Arrangement } can be defined as follows.
Consider the line arrangement $\A$ constructed in the following way: start with a regular pentagon $\PPP$ in $\R^2 \subset \C^2 \subset \PP^2$, take the 5 sides of $\PPP$ and the 5 diagonals of $\PPP$, and add the line at infinity $L_{\infty}$ to get $\A$. Then each line in $\A$ has 5 multiple points, for instance the line $L_{\infty}$ has 5 triple points, since to  each side of $\PPP$ corresponds a diagonal of $\PPP$ such that the two lines in $\R^2$ are parallel. The weak combinatorial description of this arrangement is $n_2=10$, $n_3=n_4=5$, and this implies that
$$\tau(\A)=\tau_{max}(11,5)=\tau_{max}(d,m+1).$$
Comparing with Corollary \ref{corAe1} we see that $\A$ corresponds to the case (4), and hence $\A$ is free with exponents $(5,5)$.
\end{ex}

\section{On free line arrangements with \texorpdfstring{$d_1\geq m+2$}{}}\label{sec3}

We start with the following general result that yields bounds for the number of lines of free line arrangements depending on the maximal multiplicity of the intersections.

\begin{prop}
\label{bound}
Let $\mathcal{A} \subset \mathbb{P}^{2}$ be a free arrangement of $d$ lines with $d_{1} = m+\epsilon$, where $\epsilon \geq 0$ and $m$ is the maximal multiplicity of $\A$. Then
$$2m + 2\epsilon + 1 \leq d \leq \frac{m(m+2+\epsilon)}{2}.$$
\end{prop}
\begin{proof}
The freeness of the arrangement $\mathcal{A}$ implies that $d_{1} + d_{2} =d-1$ with $d_{1}\leq d_{2}$, and we have
$$\frac{d-1}{2} \geq d_{1} = m+\epsilon,$$
which gives us $d\geq 2m + 2\epsilon +1$.
For the second inequality, recall that by \cite[Theorem 2.8]{Mich} we have
$$m+\epsilon= d_{1} \geq \frac{2}{m}\cdot d-2,$$
which gives us
$$m(m+\epsilon + 2)\geq 2d, $$
and this completes the proof.
\end{proof}

The following result shows that if we drop the freeness condition, the complexity of the line arrangement $\A$, measured by its type $t(\A)$, which is precisely defined in \cite{ADP}, may linearly increase with $\epsilon$.
\begin{prop}
\label{bound2}
Let $\mathcal{A} \subset \mathbb{P}^{2}$ be an arrangement of $d$ lines with $d_{1} = m+\epsilon$, where $\epsilon \geq -1$ and $m$ is the maximal multiplicity of $\A$. Then the following inequality holds for the type of $\mathcal{A}$, namely
$$t(\A):=d_1+d_2-d+1 \leq \epsilon +1.$$
When $\epsilon =0$, both cases $t(\A)=0$ (i.e. $\A$ is free) and $t(\A)=1$
(i.e. $\A$ is a plus-one generated arrangement) do really occur.
\end{prop}
\proof
For the first claim, it is enough to use that $d_2 \leq d-m$, see \cite{Mich}. For the second claim, use \cite[Theorem 1.12]{Der}.
\endproof

\begin{rk}
\label{rkE3}
Since the free line arrangements $\A$ with maximal multiplicity $m \leq 3$ can be easily listed, we assume in the sequel $m \geq 4$.

\medskip

\noindent (i) If $\A$ is a free arrangement with $m=4$ and $d_1=m+1=5$, Proposition \ref{bound} shows that
$$11 \leq d \leq 14.$$
Querying the database~\cite{BK} yields that in degree $d=11$ there are $7$ intersection lattices associated to free line arrangements satisfying $(m,d_1)=(4,5)$. It is interesting to note that $6$ of these intersection lattices yield divisionally free line arrangements. Indeed, in this setting,
$\A$ is divisionally free if and only if there is a line $L \in \A$ containing exactly $6$ multiple points of $\A$. Hence these $6$ line arrangements fall into the case (2) of Corollary \ref{corAe1}.
The exceptional line arrangement which is not divisionally free is the arrangement
$$\A:xyz(x+y)(x+(1+a)y)(x+z)(x+az)(y-z)(y+(1-a)z)$$
$$(x-ay+(1+a)z)(x-ay+az)=0,$$
where $a^2-a-1=0$.
Any line $L$ in this arrangement contains exactly $5$ multiple points of $\A$, and hence this arrangement is an example of the situation described in case (4) in Corollary \ref{corAe1} for {\it any choice of the line} $L \in \A$. A geometric description of this line arrangement $\A$ was given in Example \ref{exAe1} (iii) above.
\medskip

\noindent (ii)  If $\mathcal{A}$ is a free line arrangement with $d_{1} = m+3$, then $d\geq 15$. In order to see this, we observe that the following inequality must hold:
$$2m + 7 \leq \frac{m(m+5)}{2},$$
which tells us that $m\geq 4$, and then we can conclude $d\geq 15$. However, we do not know whether this lower bound is sharp, but we are aware of an example of a free line arrangement of $21$ lines satisfying $d_{1}=m+3$, namely this is the reflection arrangement $G_{26}$  with $n_{2} = 36$, $n_{3}=9$, $n_{4}=12$ and exponents $(7,13)$.
In fact, for the reflection arrangements $G_{23}$, $G_{24}$, $G_{25}$ and $G_{27}$ one has $d_1=m+ \epsilon$, where $\epsilon$ takes the values $0, 5, 0,14$, respectively. In particular, this shows that $\epsilon$ can take large values compared to $m$. The reader can find information on these reflection arrangements in \cite[Appendices B and C]{OT}.

\end{rk}

If $\epsilon =2$ and $m=4$, then by Proposition \ref{bound} we see that $13 \leq d\leq 16$.
Querying the database~\cite{BK} yields that up to $d\leq 14$ there are only two line arrangements with $d_{1} = m + 2$, one, say $\A$, of degree $d=13$ and the second, say $\CC$, of degree $d=14$. Moreover, in both cases $m=4$, and this tells us that we have found an example, namely $\A$, such that $d=13$  reaches the above lower bound in Proposition \ref{bound}. 

\subsection{On the free line arrangement $\A$ with $d_1=d_2=m+2=6$} 
First we describe a way of constructing this line arrangement starting with a free line arrangement of $7$ lines.
Consider the line arrangement
$$\A_7=\A': f'=xyz(x+y)(x+z)(y+z)(x+y+z)=0.$$
Then $\A'$ has $d'=7$ and is free with exponents $(d_1',d_2')=(3,3)$.
There are 6 triple points in $\A'$, namely
$$p_1=(1:0:0), \ p_2=(0:1:0), \ p_3=(0:0:1), \ p_4=(0:1:-1),$$
 $$ p_5 =(1:0:-1) \text{ and }
 p_6=(1:-1:0),$$
 and only 3 double points, namely 
 $$q_1=(-1:1:1), q_2= (1:-1:1) \text{ and } q_3=(1:1:-1).$$
 It follows that the maximal multiplicity $m'$ for points in $\A'$ satisfies $$m'=d_1'.$$
Let $L_j :\ell_j=0$, where $\ell_j$ is the $j$-th factor in the polynomial $f'$ above. For instance
$$L_4: \ell_4=x+y=0.$$
Consider now the group of diagonal matrices in $GL_3(\C)$
$$G=\{u=(1,e^a, e^b) \  |  \ a,b \in [0,5] \}$$
acting on $S$ in the usual way, that is
$$u\cdot g(x,y,z)=g(x,e^ay,e^bz),$$
where $e^2-e+1=0$, that is $e$ is a primitive root of unity of order 6.
Then consider the new 6 lines, obtained from the lines in $\A'$ by translations with some elements of the group $G$, namely
\begin{align*}
&L_8: \ell_8=0 \text{ with } \ell_8(x,y,z)=(1,e^2,1)\cdot \ell_4(x,y,z)=x+e^2y,\\
&L_9: \ell_9=0 \text{ with } \ell_9(x,y,z)=(1,1,e)\cdot \ell_5(x,y,z)=x+ez,\\
&L_{10}: \ell_{10}=0 \text{ with } \ell_{10}(x,y,z)=(1,e,1)\cdot \ell_6(x,y,z)=ey+z,\\
&L_{11}: \ell_{11}=0 \text{ with } \ell_{11}(x,y,z)=(1,e,e)\cdot \ell_7(x,y,z)=x+ey+ez,\\
&L_{12}: \ell_{12}=0 \text{ with } \ell_{12}(x,y,z)=(1,e,1)\cdot \ell_7(x,y,z)=x+ey+z,\\
&L_{13}:\ell_{13}=0 \text{ with } \ell_{13}(x,y,z)=(1,e^2,e)\cdot \ell_7(x,y,z)=x+e^2y+ez.
\end{align*}
Then the line arrangement
$$\A_{13}=\A:f=f' \cdot \ell_8\ell_9\ell_{10}\ell_{11}\ell_{12}\ell_{13}=0
$$
has $d=13$ and it is free with exponents $(6,6)$, as one can see using a direct computation, for instance using \verb}SINGULAR} \cite{Sing}. 
Note that $\A$ has $7$ points of multiplicity $4$, namely
$$(1:0:0), \ (0:1:0), \ (0:0:1), \ (0:1:-1), \ (1:0:-1), \ (e:0:-1), \ (0:-1:e).$$
In addition, it has $9$ triple points and $9$ double points. Let us observe now an interesting extremal combinatorial property of the arrangement $\A=\A_{13}$.
\begin{prop}[cf. {\cite[Proposition 6.5]{Osaka}}]
\label{propCOMB}
Let $\A \subset \mathbb{P}^{2}$ be a free arrangement of $d$ lines. Then
$$\sum_{r\geq 2}(r-1)n_{r} \leq \bigg\lfloor \frac{(d-1)(d+3)}{4}\bigg\rfloor,$$
and this bound is sharp, for instance it is achieved by $\mathcal{A}_{13}$.

\end{prop}
\begin{proof}
Recall that the freeness of $\A$ implies that 
$$d_{1}d_{2} = d_{1}(d-d_{1}-1) = \sum_{r\geq 2}(r-1)n_{r}-d+1,$$
which gives us 
$$-d_{1}^{2}+d_{1}(d-1) - \bigg(\sum_{r\geq 2}(r-1)n_{r}-d+1\bigg)=0.$$
The above condition implies that the discriminant satisfies
$$\triangle_{d_{1}} = (d-1)^2 - 4\bigg(\sum_{r\geq 2}(r-1)n_{r}-d+1\bigg) \geq 0,$$
and this gives us
$$\sum_{r\geq 2}(r-1)n_{r} \leq \bigg\lfloor \frac{(d-1)(d+3)}{4}\bigg\rfloor.$$
Then it is easy to check that for $\A_{13}$ we do indeed get equality.
\end{proof}

\begin{rk}
\label{rkA13}
We also notice the following homological properties of the arrangements constructed as additions to $\A'$.
\begin{enumerate}
\item $\A_{8}= \A' \cup L_8$ is nearly free with exponents $(4,4,4)$.
\item $\A_{9}=\A' \cup L_8 \cup L_9$ is a 4-syzygy curve with exponents $(5,5,5,5)$. This curve is in fact a maximal Tjurina curve of type $(d,r)=(9,5)$, see \cite{maxT} for the definition and the properties of such curves. 
Moreover, it is known that the defect of this curve is $\nu(\A_9)=2$, see \cite[Theorem 3.11]{3syz}.
 
\item $\A_{10}=\A' \cup L_8 \cup L_9 \cup L_{10}$ is a 4-syzygy curve with exponents $(5,6,6,6)$. A \verb}SINGULAR} computation shows that the defect of this curve is equal to $\nu(\A_{10})=3$. 
 
\item $\A_{11}=\A' \cup L_8 \cup L_9 \cup L_{10} \cup L_{11}$ is a 4-syzygy curve with exponents $(6,6,6,6)$. This curve is in fact a maximal Tjurina curve of type $(d,r)=(11,6)$, see \cite{maxT}. Moreover, we can check that the defect is equal to $\nu(\A_{11})=2$, see \cite[Theorem 3.11]{3syz}.
 
\item $\A_{12}=\A' \cup L_8 \cup L_9 \cup L_{10} \cup L_{11} \cup L_{12}$ is nearly free with exponents $(6,6,6)$. 
\end{enumerate} 
Let us observe that all these $5$ arrangements and $\A$ have maximal multiplicity of points
 $m=4$, e.g., the point $(0:0:1)$ is of multiplicity 4 for all of them.
 Hence the first exponent $d_1$ for them takes all values in the interval
 $$[m,m+2].$$
 Moreover, by adding lines to the free line arrangement $\A_7$, with type $t(\A_7)=0$ according to Definition 1.2 in \cite{ADP}, we get, in ascending order, line arrangements with types $1,2,2,2,1$, hence the homological complexity first increases and then decreases.
 \end{rk}
 We would like to prove geometrically that $\A$ is free, that is without computer-aided computations. As we have seen in Remark \ref{rkA13},
if we construct $\A$ starting from $\A'$ and by adding $6$ lines, we go far away from the class of free arrangements in this process. Indeed, free arrangements correspond to the defect equal to $\nu=0$, nearly free arrangements have defect equal to $\nu=1$, but the arrangement $\A_{10}$ has defect equal to $\nu=3$. That is why we look for an alternative construction of $\A$ starting with a free line arrangement with $11$ lines to be described below.

Let $L_j$ be the line defined by the $j$-th factor in the equation for $\A_{13}$, for $j=1,\ldots,13$.
Then we see the following.
\begin{itemize}
\item[I1.] The lines $L_1$ and $L_2$ contain each $4$ points of multiplicity $4$.

\item[I2.] The lines $L_3, L_5,L_6,L_7,L_9,L_{10}, L_{11},L_{12} $ and $L_{13}$ contain each $2$ points of multiplicity $4$.

\item[I3.] The lines $L_4$ and $L_8$ contain each a single point of multiplicity $4$.
\end{itemize}
To determine the number of points of multiplicity $4$ on a given line $L_j$, we look at sets $\{L_j,L_a,L_b,L_c\}$ of 4 lines in $\A$, one of them being $L_j$, such that the corresponding equations $\{\ell_j,\ell_a,\ell_b,\ell_c\}$ span a 2-dimensional vector subspace in $S_1$. For instance, one point of multiplicity $4$ on $L_1$ comes from the set $\{L_1,L_2,L_4,L_8\}$,
since the corresponding equations $\{x,y,x+y,x+e^2y\}$ clearly span a 2-dimensional subspace in $S_1$, in other words they form a pencil of linear forms.

It follows that, in some sense, the lines $L_4$ and $L_8$ are the \textit{worst lines} in the arrangement $\A_{13}$. We have the following results.

\begin{thm}
\label{thm1}
Let $\B_{11}$ be the arrangement obtained from $\A_{13}$ by deleting the lines $L_4$ and $L_8$. Then $\B_{11}$ is a free arrangement with exponents $(4,6)$.
\end{thm}
\proof
The defining equation of $\B_{11}$ is
$$g_{11}=xyz(x+z)(y+z)(x+y+z)(x+ez)(ey+z)(x+ey+ez)(x+ey+z)(x+e^2y+ez)=0.$$
Note that $(1:0:0)$ is a point of multiplicity 4 for  $\B_{11}$, since the lines
$y=0$, $z=0$, $y+z=0$ and $ey+z=0$ meet there. It follows that
the minimal exponent $d_1$ of $\B_{11}$ satisfies
$$d_1 \geq \bigg\lceil\frac{1}{2}\cdot 11 - 2 \bigg\rceil = 4,$$
see \cite[Theorem 2.8]{Mich}.
If $d_1=4$, then $ \B_{11}$ is free with exponents $(4,6)$, which gives the correct Tjurina number, and $d_1 >4$ yields too small Tjurina numbers for $\B_{11}$, see also \cite[Remark 5.4]{Der}.
\endproof

To continue this analysis, we recall the following exact sequence of $\C$-vector spaces which is a reformulation of \cite[Theorem 6.2]{DIS}:
\begin{equation}
\label{eq1}
0 \to D_0(g')_{k-1} \to D_0(g)_k \to R_{k+1-r} \to N(g')_{k+d-3} \to N(g)_{k+d-1} \to R_{r-k-3}
\end{equation}
where $\B:g=0$ is an arrangement of $d$ lines in $\PP^2$, $\B':g'=0$ is the arrangement obtained from $\B$ by deleting one line, say $L \in \B$,
$k$ is any integer, $R=\C[y,z]$ is the polynomial ring in $y,z$, we set $r= |L \cap \B'|$, and $N(g)$, $N(g')$ are the graded Jacobian $S$-modules corresponding to $g$ and $g'$, respectively -- see \cite{DIS} for the necessary definition. Using this exact sequence we can prove the following.

 \begin{thm}
\label{thm2}
The arrangement  $\A=\A_{13}$ is free  with exponents $(6,6)$.
\end{thm}
\proof

Note that $\A$ is obtained from  $\B_{11} $ by adding the lines $L_4$ and $L_8$. The intersection $L_4 \cap  \B_{11}$ consists of 7 points, so 
if we apply the exact sequence \eqref{eq1} to the pair
$(\B,\B')=(\B_{12}=\B_{11}  \cup L_4, \B_{11} )$ for $k=4$, and we get
$$0 \to D_0(g')_3 \to D_0(g)_4 \to R_{-2}$$
and hence $D_0(g)_4=0$. Now the same exact sequence for $k=5$ gives
$$0 \to D_0(g')_4 \to D_0(g)_5 \to R_{-1}$$
and hence $\dim D_0(g)= \dim D_0(g')=1$. It follows that the first exponent of the arrangement $\B_{12}$ is $r_{12}=5$.

Now we add the line $L_8$ to $\B_{12}$ to get the arrangement $\A$.
The intersection $L_8 \cap  \B_{12}$ consists of $7$ points, so 
if we apply the exact sequence \eqref{eq1} to the pair
$(\B,\B')=(\A, \B_{12} )$ for $k=5$ and we get
$$0 \to D_0(g')_4 \to D_0(g)_5 \to R_{-1}$$
and hence $D_0(g)_5=0$. 
It follows that the first exponent of the arrangement $\A$ is $r_{13} \geq 6$.
Since
$$\tau(\A)=12^2-6^2=108= \tau_{max}(13,6),$$
it follows that $r_{13} = 6$ and $\A$ is free with exponents $(6,6)$ as we have claimed.
\endproof

The proofs above give perhaps a new proof for the following known result.

 \begin{cor}
\label{cor2}
The arrangement $\A=\A_{13}$ satisfies Terao's Conjecture.
\end{cor}
Indeed, for any line arrangement $\overline \A$, having the same intersection lattice with $\A$, the proofs of Theorems \ref{thm1} and \ref{thm2} can be applied to yield this claim.

\subsection{On the free line arrangement $\CC$ with $d_1=m+2=6<d_2=7$} 

Consider now the line arrangement
$$\CC_{14} =xyz(x+y)(x+z)(x+(1-e)y)(x+(e-1)z)(y+(e-1)z)(x+(2-e)y+z)$$
$$(x+y+(e-1)z)(y+(e-2)z)(x+(1-e)y+z)(x+(2-e)y+(e-1)z)(x+(2-e)y+ez)=0,$$
where $e^2-3e+3=0$.
This arrangement has $n_2=13$, $n_3=6$ and $n_4=10$. Moreover it is free with exponents $(6,7)$. Now we prove its freeness geometrically. If we look at the number of points of multiplicity $4$ situated on each line in $\CC_{14}$, we see that the lines $L_4$ and $L_6$, corresponding to the 4th and the 6th factors above, are exceptional because each of them contains only one point of multiplicity $4$.
As a result, we consider the  arrangement $\D_{12}$ obtained from $\CC_{14}$ by deleting the lines $L_4$ and $L_6$.

 \begin{thm}
\label{thm10}
Let $\D=\D_{12}$ be the arrangement obtained from $\CC=\CC_{14}$ by deleting the lines $L_4$ and $L_6$. Then $\D$ is a free arrangement with exponents $(4,7)$.
\end{thm}
\proof
The proof goes along the same lines as for Theorem \ref{thm1}. First of all, we observe that $(0:1:0)$ is a point of multiplicity $4$ for $\D$ since the lines $x=0$, $z=0$, $x+z=0$ and $x+(e-1)z=0$ meet there. It follows that
the minimal exponent $d_1$ of $\D$ satisfies
$$d_1 \geq \bigg\lceil\frac{1}{2}\cdot 12 - 2 \bigg\rceil = 4,$$
see \cite[Theorem 2.8]{Mich}.
If $d_1=4$, then $ \D$ is free with exponents $(4,7)$, which gives the correct Tjurina number, and $d_1 >4$ yields too small Tjurina numbers for $\D$, see also \cite[Remark 5.4]{Der}.
\endproof

 \begin{thm}
\label{thm20}
The arrangement $\CC=\CC_{14}$ is a free arrangement with exponents~$(6,7)$.
\end{thm}
\proof
To get from $\D$ to $\CC$ we add first the line $L_4: x+y=0$ and get a new arrangement $\CC'$. Then one has $|L_4 \cap \D|=8$, and the exact sequence \eqref{eq1} applied to the pair $(\CC',\D)$ and to $k=4$ yields
$$0 \to D_0(g')_3 \to D_0(g)_4 \to R_{-3}$$
which gives us $D_0(g)_4=0$ since $D_0(g')_3=R_{-3}=0$.
The same exact sequence for $k=5$ yields
$$1= \dim D_0(g')_3 =\dim D_0(g)_4 ,$$
since $R_{-2}=0$. It follows that the minimal exponent of $\CC'$ is 5.

Then we add to $\CC'$ the line $L_6: x+(1-e)y=0$.  One has $|L_6 \cap \CC'|=8$ and the exact sequence \eqref{eq1} applied to the pair $(\CC,\CC')$ and to $k=5$ yields
$$0 \to D_0(g')_4 \to D_0(g)_5 \to R_{-2}$$
which gives us $D_0(g)_5=0$ since $D_0(g')_4=R_{-2}=0$. It follows that the minimal exponent $d_1$ of $\CC$ satisfies $d_1 \geq 6$. On the other hand, we have
$$\tau(\CC)=127= \tau_{max}(14,6),$$
which implies that $\CC$ is free with exponents $(6,7)$.

\endproof

The proofs above give a new proof for the following known result.

 \begin{cor}
\label{cor20}
The arrangement $\CC=\CC_{14}$ satisfies Terao's Conjecture.
\end{cor}
Indeed, for any line arrangement $\overline \CC$, having the same intersection lattice with $\CC$, the proofs of Theorems \ref{thm10} and \ref{thm20} can be applied to yield this claim.

\begin{rk}
\label{rkDF}
    Both arrangements $\A$ and $\CC$ are divisionally free~\cite{DivFree}. Indeed the characteristic polynomial $\chi(\A,t)$ is 
    $(t-1)(t-6)^2$ and the characteristic polynomial of the restriction $\A^L$ of $\A$ to the line $L=L_8$ considered in the proof of Theorem \ref{thm2} is $\chi(\A^L,t)=(t-1)(t-6)$. Similarly, the characteristic polynomial $\chi(\mathcal{C},t)$ is 
    $(t-1)(t-6)(t-7)$ and the characteristic polynomial of the restriction $\CC^L$ of $\CC$ to the line $L=L_6$ considered in the proof of Theorem \ref{thm20} is $\chi(\A^L,t)=(t-1)(t-7)$. Since $\chi(\A,t)$ is divisible by
    $\chi(\A^L,t)$ and respectively $\chi(\CC,t)$ is divisible by
    $\chi(\CC^L,t)$, it follows that both line arrangements $\A$ and $\CC$ are divisionally free.
This fact also, in particular, proves that they are free.
As being divisionally free is a combinatorial property depending on the intersection lattice only, this also yields a second proof of Corollary~\ref{cor2} and~\ref{cor20}. 

In fact, for a line arrangement $\A$ in $\PP^2$ being divisionally free is equivalent to the existence of a line $L \in \A$ such that $\A$ and deleted arrangement $\A'=\A \setminus L$ are both free, see \cite[Theorem  3.11]{DivFree}. In particular, we see in this way that the arrangement $\B_{12}=\A \setminus L_8$ in the proof of Theorem \ref{thm2} and the arrangement $\CC'=\CC \setminus L_6$ in the proof of Theorem \ref{thm20} are both free.
        
On the other hand, note that the arrangement $\A$ in Theorem \ref{thm02} (2) and the arrangement $\B$ in Theorem \ref{thm03} (2)
are also divisionally free. However, the monomial arrangements in Example \ref{ex03} are not divisionally free and  satisfy $d_1=m+1$.
More generally, the arrangement
$$\A^0_3(r): \prod_{0\leq n <r}(x-\zeta^ny)(x-\zeta^nz)(y-\zeta^nz)=0 $$
satisfies $d_1=m+1$ and is not divisionally free, see \cite[Theorem 5.6]{DivFree}.

It would be interesting to find free line arrangements with minimal number of lines which are not divisionally free and which satisfy  $d_1=m+2$.

\end{rk}

\section*{Acknowledgement}

Alexandru Dimca is partially supported from the project ``Singularities and Applications'' - CF 132/31.07.2023 funded by the European Union - NextGenerationEU - through Romania's National Recovery and Resilience Plan.

Lukas K\"uhne is funded by the Deutsche Forschungsgemeinschaft (DFG, German Research Foundation) – Project-ID 491392403 – TRR 358 and SPP 2458 -- 539866293.

Piotr Pokora is supported by the National Science Centre (Poland) Sonata Bis Grant \textbf{2023/50/E/ST1/00025}. For the purpose of Open Access, the author has applied a CC-BY public copyright license to any Author Accepted Manuscript (AAM) version arising from this submission.

\section*{Conflict of Interests}
We declare that there is no conflict of interest regarding the publication of this paper.
\section*{Data Availability Statement}
We do not analyse or generate any datasets, because this work proceeds within a theoretical and mathematical approach.

\end{document}